\documentclass{amsart}

\usepackage{amssymb, amsmath,amsthm}
\usepackage{braket}

\newtheorem{theorem}{Theorem}[section]
\newtheorem{fact}[theorem]{Fact}
\newtheorem{question}[theorem]{Question}
\newtheorem{proposition}[theorem]{proposition}
\newtheorem{corollary}[theorem]{Corollary}
\newtheorem{remark}[theorem]{Remark}
\newtheorem{lemma}[theorem]{Lemma}

\theoremstyle{definition}
\newtheorem{definition}[theorem]{Definition}
\newtheorem{notation}[theorem]{Notation}

\newcommand{\bQ}{\mathbb{Q}}

\newcommand{\iLs}{$ \mathcal{L} \ $} 
\newcommand{\iMs}{$ \mathcal{M} \ $} 
\newcommand{\iNs}{$ \mathcal{N} \ $}

\newcommand{\iL}{$ \mathcal{L} $} 
\newcommand{\iM}{$ \mathcal{M} $} 
\newcommand{\iN}{$ \mathcal{N} $}

\newcommand{\La}{\mathcal{L}} 
\newcommand{\M}{\mathcal{M}} 
\newcommand{\N}{\mathcal{N}}

\newcommand{\embedS}{\lesssim_{s}}

\newcommand{\sims}{\sim_{s}}
\newcommand{\embedArrow}{\hookrightarrow}

\newcommand{\G}{\Gamma}
\newcommand{\Ga}{\Gamma^*}
\newcommand{\iG}{$\Gamma$}
\newcommand{\iGa}{$ \Gamma^*$}
\newcommand{\iGs}{$ \Gamma\ $}
\newcommand{\iGas}{$ \Gamma^*\ $}

\newcommand{\cont}{2^{\aleph_0}}
\newcommand{\Q}{\mathbb{Q}}

\newcommand{\cM}{\mathcal{M}}
\newcommand{\cN}{\mathcal{N}}
\newcommand{\cL}{\mathcal{L}}
\newcommand{\cC}{\mathcal{C}}

\newcommand{\cG}{\mathcal{G}}

\newcommand{\twopartdef}[4]
{
	\left\{
		\begin{array}{ll}
			#1 & \mbox{if } #2 \\
			#3 & \mbox{if } #4
		\end{array}
	\right.
}

\DeclareMathOperator{\age}{age}

\begin{document}

\pagespan{3}{}

\title{Many symmetrically indivisible structures}

\author[N. Meir]{Nadav Meir}\email{mein@math.bgu.ac.il}
\address{Department of Mathematics, Ben-Gurion University of the Negev, P.O.B. 653, Be'er Sheva 8410501, ISRAEL}




\begin{abstract}
A structure $\cM$ in a first-order language $\cL$ is \emph{indivisible} if for every colouring of its universe $M$ in two colours, there is a monochromatic $\M^{\prime} \subseteq \M$ such that $\M^{\prime}\cong\cM$. Additionally, we say that $\mathcal{M}$ is \emph{symmetrically indivisible} if $\mathcal{M}^{\prime}$ can be chosen to be \emph{symmetrically embedded} in $\mathcal{M}$ (that is, every automorphism of $\mathcal{M}^{\prime}$ can be extended to an automorphism of $\mathcal{M}$). In the following paper we give a general method for constructing new symmetrically indivisible structures out of existing ones. Using this method, we construct $2^{\aleph_0}$ many non-isomorphic symmetrically indivisible countable structures in given (elementary) classes and answer negatively the following question from \cite{HKO11}: Let \iMs be a symmetrically indivisible structure in a language
\iLs. Let $\La_0 \subseteq \La$. Is $ \M \upharpoonright \La_0$
symmetrically indivisible?
\end{abstract}


\maketitle



\section{Introduction}
The notion of indivisibility of relational first-order structures and metric spaces is well studied
in Ramsey theory (\cite{KoRo86}, \cite{ElSa91},\cite{ElSa93} are just a few examples of the extensive study in this area). Recall that a
structure $\cM$ in a relational first-order language is indivisible, if for every colouring of its universe $M$ in two colours, there is a monochromatic substructure $\cM^{\prime} \subseteq \cM$ such that $\cM^{\prime} \cong \cM$. Rado's random graph, the ordered set of
natural numbers and the ordered set of rational numbers are just a few of the many
examples. Weakenings of this notion  have also been studied (see \cite{Sa14}). A known extensively studied strengthening of this notion is the pigeonhole property. A first-order relational structure $X$ admits the pigeonhole property if whenever $X$ is the union of two disjoint substructures $Y$ and $Z$, at least one of $Y$ and $Z$ is isomorphic to $X$. Examples of such structures include the random graph and the random $n$-hypergraph, though in general such structures are very rare. (See \cite{Ca10} for further reading.)
For an extensive review on the subject --- see appendix A in \cite{Fr00}. 

In \cite{GeKo11}, a notion of symmetrized Ramsey theory was introduced, and in \cite{HKO11} a new strengthening of the notion of indivisibility was investigated: We say that a substructure $\cN \subseteq \cM$ is \emph{symmetrically embedded} in \iMs if every automorphism of \iNs extends to an automorphism of $\cM$. We say that \iMs is \emph{symmetrically indivisible} if for every colouring of $M$ in two colours, there is a monochromatic $\M^{\prime} \subseteq \M$ such that $\M^{\prime}$ is isomorphic to \iMs and $\M^{\prime}$ is symmetrically embedded in $\cM$.

In \cite{HKO11}, several examples of symmetrically indivisible structures were  introduced. Examples include the random graph (\cite{GeKo11}), the ordered rational numbers, the ordered natural numbers, the universal $n$-hypergraph.

 In Section \ref{many} we will present a method of constructing new symmetrically indivisible structures out of existing ones and using this method, construct $2^{\aleph_0}$ many non-isomorphic symmetrically indivisible countable linear orders and $2^{\aleph_0}$ many non-isomorphic symmetrically indivisible countable graphs. We give a sufficient conditions for a class of $\cL$-structures to have $2^{\aleph_0}$ many symmetrically indivisible structures. We note that these conditions are met by the class of $n$-hypergraphs, graphs edge-coloured in $k\leq \omega$ colours, and trivially by the class of partial orders as a super-class of linear orders and with an aim of finding a general claim. 

In Section \ref{sec3} we make further use of this method to construct an example that answers negatively a question asked in \cite{HKO11}: Let \iMs be a symmetrically indivisible structure in a language
\iLs. Let $\La_0 \subseteq \La$. Is $ \M \upharpoonright \La_0$
symmetrically indivisible?
It is clear that if $\cM$ is indivisible then $\cM\upharpoonright \cL_0$ is indivisible, but for a symmetrically embedded $\cM_0 \subseteq \cM$, $\cM_0 \upharpoonright \cL_0$ is not necessarily symmetrically embedded in $\cM$, thus this question does not seem to have an immediate answer.
\section{Many examples for symmetrically indivisible structures}\label{many}
In this section we will show how to construct many symmetrically indivisible structures based on existing ones.

\begin{definition}
	For two first-order structures $\cM$, $\cN$ an embedding $e:\cN\to \cM$ is called \emph{symmetric} if every automorphism of $e[\cN]$ extends to an automorphism of $\cM$. So a substructure $\cN\subset \cM$ is symmetrically embedded if the inclusion map $\iota$ is symmetric.
\end{definition}

\begin{notation}
	Let \iL\ be a first-order language and let \iM , \iN\ be \iL -structures.
We write $\M \lesssim \N $ if there is an embedding $ e:\M \embedArrow \N $. In the same fashion we write $\M \lesssim_{s} \N $ if there is a \emph{symmetric} embedding $ e:\M \embedArrow \N $.

Similarly, we write  $\M \subseteq \N$ if $\M$ is a  substructure of $\N$ and $\cM \subseteq_s \cN$ if there is a \emph{symmetric} embedding $e:\cM\embedArrow \cN$.

Finally, $\M \sim \N $ means that both $ \M \lesssim \N $ and $ \N \lesssim \M $ hold, and $\cM \sim_s \cN$ means that both $ \M \lesssim_s \N $ and $ \N \lesssim_s \M $ hold.
\end{notation}

\begin{proposition}\label{equivResult}\ 
\begin{enumerate}
\item If \iMs and \iNs are \iL-structures such that $ \M \sim \N $ then
 \iMs is  indivisible iff \iNs is  indivisible.
\item
If \iMs and \iNs are \iL-structures such that $ \M \sim_{s} \N $ then
 \iMs is \emph{symmetrically  } indivisible iff \iNs is \emph{symmetrically} indivisible.
\end{enumerate}
	
\end{proposition}

\begin{proof}\ 
\begin{enumerate}
\setcounter{enumi}{1}
\item 
	Because $ \sim_{s} $ is an equivalence relation, it is enough to show one direction:
	suppose \iM\ is symmetrically indivisible, and let $ c:\N\rightarrow\{\text{red},\text{blue}\} $ be a colouring of \iN. Since $\M \lesssim_{s} \N$, let $\M_0 \subseteq_{s} \N$ be   such that $\M_0 \cong \M$, so $c \upharpoonright \M_0$ is a colouring of $\M_0$, and since \iMs is symmetrically indivisible, so is $\M_0$ and there is a monochromatic $\M_0^{\prime} \subseteq_{s} \M_0$ such that $\M_0^{\prime}$ is isomorphic to $\M_0 \cong \M$. Now, since $\N \lesssim_{s} \M \cong \M_0^{\prime}$,  there is $\N_0 \subseteq_{s} \M_0^{\prime} $ such that $\N_0 $ is isomorphic to \iNs and since $\M_0^\prime$ is monochromatic, so is $\N_0$. Now, since $\N_0 \subseteq_{s} \M_0^{\prime} \subseteq_{s} \M_0 \subseteq_{s} \N $, by transitivity $\N_0 \subseteq_{s} \N $.
\setcounter{enumi}{0}
\item 
	Repeat the same argument, omitting ``symmetric".
\end{enumerate}

\end{proof}

Now we are ready to construct $\cont$ examples of symmetrically indivisible graphs and $\cont$ examples of symmetrically indivisible linear orders, both based on known symmetrically indivisible structures, and the equivalence relation $\sims$:

Recall:

\begin{fact}\label{RadoQSymIndivisible}
the random graph and $(\Q,<)$ are  symmetrically indivisible. (\cite{GeKo11}, \cite{HKO11})
\end{fact}

In \cite{Hen71} it was shown that:
\begin{fact}\label{gammaEmbeds}
Let $\Gamma$ be the random graph. For every countable graph $G$, $G \embedS \G$.
\end{fact}
\begin{corollary}\label{embedsGtheSym}
Every countable graph which symmetrically embeds \iGs is symmetrically indivisible.
\end{corollary}
\begin{proof}
Let $G$ be a countable graph which symmetrically embeds \iG . Then by definition $G \sims \G$ thus by Fact \ref{RadoQSymIndivisible} combined with Proposition \ref{equivResult}, $G$ is symmetrically indivisible.
\end{proof}
For $(\Q,<)$ we have a result similar to Fact \ref{gammaEmbeds}:
\begin{proposition}\label{QEmbeds}
For every countable linear order $A$, $A \embedS \Q$.
\end{proposition}
\begin{proof}
Let $A[\Q]$ be the lexicographic order on $A\times \bQ$.
This is a countable dense linear order without end-points (DLO). By $\aleph_0$-categoricity of DLO, it is isomorphic to $\langle \bQ, <\rangle$.

For a fixed $q\in \bQ$, the induced substructure on $A\times \{q\}$ is isomorphic to $A$ and the fact that it is symmetrically embedded can be easily verified and is actually a special case of Lemma 2.8 of \cite{HKO11}.
\end{proof}
\bigskip

\noindent From this we have, exactly like Corollary \ref{embedsGtheSym} for graphs :
\begin{corollary}\label{embedQthenSym}
Every countable linear order which symmetrically embeds $(\Q,<)$ is symmetrically indivisible.
\end{corollary}

\begin{definition}
Let $G$, $H$ be graphs and  for convenience assume $|G|\cap|H|=\emptyset$. We define $G+^{\cG}H$ to be  the graph whose universe is $|G|\cup|H|$ and $E^{(G +^{\cG} H)} := E^G \cup E^H$.

Namely, $G+^{\cG}H$ is just the ``disjoint union of graphs" as known in graph theory and denoted by  $G\cup H$.
\end{definition}
\begin{remark}\label{remarkGraphAddition} Let $G,H,K$ be graphs.
\begin{enumerate}
\item $G,H\lesssim_{s} G+^{\cG} H$
\item $G+^{\cG} H = H+^{\cG} G $
\item $(G+^{\cG} H) +^{\cG} K = G+^{\cG} (H+^{\cG} G) $
\item If $C\subseteq G$ is a union of connected components, then $G = (G\setminus C)+^{\cG} C $
\end{enumerate}
\end{remark}
\begin{proposition}\label{lotsofGraphs}
If $G$, $H$ are graphs  such that $\G +^{\cG} G \cong \G +^{\cG} H$, then $G \cong H$.
\end{proposition}
\begin{proof}
Let $\phi: \G +^{\cG} G \to \G +^{\cG} H$ be an isomorphism. Since $\phi$ maps connected components onto connected components and $\G$ is connected, either $\phi[\G] = \G$ or $\phi[\G]\subseteq H$. In the first case $\phi [G] = H$ and $\phi\upharpoonright G:G \to H$ is an isomorphism. In the second case,   by Remark  \ref{remarkGraphAddition}
$$H+^{\cG} \G = (H\setminus \phi[\G])+^{\cG} \phi[\G]+^{\cG} \G$$
thus $\phi\upharpoonright G: G\to (H\setminus \phi[\G])+^{\cG} \G$ is an isomorphism, but $(H\setminus \phi[\G])+^{\cG} \G \cong H$.
\end{proof}

\begin{definition}
Let $A$ and $B$ be linear orders, and  for convenience assume $|A|\cap|B|=\emptyset$. We define $A+^{lo}B$ the linear order whose universe is $|A|\cup|B|$ and \[<^{A +^{lo} B} := <^A \cup <^B \cup\ \big\{(a,b)\ |\ a\in A \text{ and } b\in B \big\}.\] Namely, $A+^{lo}B$ is the ``concatenation" of $A$ and $B$ -- just putting $B$ right after $A$.
\end{definition}

\begin{proposition}\label{lotsofOrders}
If $X$ is a linear order such that $|X|=\{x,y\}, <^X = {(x,y)}$, and $A$, $B$ are linear orders such that $\big( \Q +^{lo} X +^{lo} A \big) \cong \big(\Q +^{lo} X +^{lo} B \big)$, then $A \cong B$.
\end{proposition}
\begin{proof}
Let $\phi:\big( \Q +^{lo} X +^{lo} A \big) \rightarrow \big( \Q +^{lo} X +^{lo} B\big)$ be an isomorphism. Now  $x$ is the minimal element with an immediate successor and $y$ is the immediate succesor of $x$ in both losets, thus $\phi(x) = x$, and $\phi(y) = y$. Thus \[\phi[A] = \phi\Set{z\in \Q +^{lo} X +^{lo} A | y<z} =  \Set{z\in \Q +^{lo} X +^{lo} B | y<z} = B\]
\end{proof}
Similarly to Remark \ref{remarkGraphAddition}:
\begin{remark}\label{remarkLoAddition}
 If $A$, $B$ are linear orders, then $A, B \embedS \big(A +^{lo} B\big)$

\end{remark}

\begin{corollary}\label{manyBoth}\ 
\begin{enumerate}
\item There is a 1-1 map between isomorphism classes of countable graphs and isomorphism classes of countable symmetrically indivisible graphs.
\item There is a 1-1 map between isomorphism classes of countable \emph{losets} and isomorphism classes of countable symmetrically indivisible \emph{losets}.
\end{enumerate}
\end{corollary}

\begin{proof}Consider the maps:\ 
\begin{enumerate}
\item $G\mapsto \G+^{\cG} G $
\item $A\mapsto \bQ+^{lo}\{x\}+^{lo}\{y\} +^{lo} A $
\end{enumerate}
By Propositions  \ref{lotsofGraphs} and \ref{lotsofOrders} these are 1-1. By Remark  \ref{remarkGraphAddition} and Corollary \ref{embedsGtheSym}, all the graphs of the form $\G+^G G$ are symmetrically indivisible.
By Remark  \ref{remarkLoAddition} and Corollary  \ref{embedQthenSym} all the losets of the form $\bQ+^{lo}\{x\}+^{lo}\{y\} +^{lo} A $ are symmetrically indivisible.
\end{proof}


We conclude this section with an attempt to generalize both constructions. As mentioned in the introduction, classic examples of symmetrically indivisible structure, in addition to the random graph and the ordered rational numbers, include the universal $n$-hypergraph and the universal edge-coloured graph in $k\leq \omega$ many colours defined below:

\begin{definition}
For $k\leq \omega$, an edge-coloured graph in $k$ many colours $G$ is a graph whose edges are coloured in $k$ many colours -- i.e. it is a structure in the language $\cL_k:=\{R_i\}_{i\in k}$ such  that $\{G \upharpoonright R_i\}_{i\in k}$ are edge-disjoint graphs.

For a fixed $k$, the class of finite edge-coloured graphs in $k$ many colours is a Fra\"iss\'e  class. We denote its Fra\"iss\'e  limit by $\Gamma_k$.
\end{definition}
 
In an attempt to generalize Corollary \ref{manyBoth}, we haven't had much success in  giving an interesting generalization other than the trivial one, which  goes as follows:
\begin{proposition}\label{generalAttempt}
Assume $\cC$ is a class of countable structures in a fixed language $\cL$, along with a symmetrically indivisible structure $\cM \in \cC$ such that for every $C\in \cC$, $C\lesssim_{s} \cM$ and a binary operation on $\cC$, $+^{\cC}$ satisfying $\cM +^{\cC} C_1 \cong \cM +^{\cC} C_2 \implies C_1\cong C_2$ and $A,B\lesssim_s A+^{\cC} B$. Then there is a 1-1 map between structures of $\cC$ up to isomorphism and symmetrically indivisible structures of $\cC$ up to  isomorphism.
\end{proposition}

Regarding $n$-hypergraphs and edge-coloured graphs: A similar construction to that of Theorem 3.1 in \cite{Hen71} can give us a result similar to Fact \ref{gammaEmbeds} and Proposition \ref{QEmbeds}: 

\begin{enumerate}
\item The universal $n$-hypergraph symmetrically embeds every $n$-hypergraph.
\item $\Gamma_k$ symmetrically embeds every edge-coloured graph in $k$ colours.
\end{enumerate}
For $n$-hypergraphs and edge-coloured graphs, $+$ will be just the disjoint union, similar to $+^{\cG}$. Proposition \ref{generalAttempt} gives us a 1-1 map between $n$-hypergraphs up to  isomorphism and symmetrically indivisible hypergraphs up to isomorphism and the same for edge-coloured graphs in $k\leq \omega$ colours.

\section{A Symmetrically Indivisible structure with a reduct that is not Symmetrically Indivisible}\label{sec3}

Recall that in \cite{HKO11} the following question was  asked:
\begin{question}\label{questionReduct}
Let \iMs be a symmetrically indivisible structure in a language
\iLs. Let $\La_0 \subseteq \La$. Is $ \M \upharpoonright \La_0$
symmetrically indivisible?
\end{question}

In this section, we will construct an example answering this question negatively.

First we construct an indivisible structure which is not symmetrically indivisible. The existence of such a structure is a necessary condition for the existence of an example for Question \ref{questionReduct}, since if $\cM$ is indivisible (in particular if it is symmetrically indivisible) then $\cM \upharpoonright \cL_0$ is also indivisible.

\subsection{$\Gamma^*$ -- an example of an indivisible structure which is not symmetrically indivisible}\

Throughout this subsection  $\Gamma$ will denote the random graph. The indivisibility of the random graph is a well known fact that dates back to its definition in \cite{Ra64}. An easy proof is given in \cite{Hen71}.

\begin{definition}\label{GammaStar}
We define the graph $\Gamma^*$ as follows:

Let $ \{K_n\}_{n<\omega}$ be disjoint sets, satisfying $|K_n|=n$ and let $\{g_n\}_{n<\omega}$ be an enumeration of $\Gamma$.
The universe of $\Gamma^{*}$ is defined to be $|\Gamma^{*}| = |\Gamma|\cup \bigcup_{n<\omega}K_n$ and the edges are defined as follows:
\[E^{\Gamma^{*}} = E^{\Gamma} \cup \bigcup_{n<\omega} \Set{(a,b) | a,b \in K_n \cup \{g_n\},\ a\neq b}\]
\end{definition}
In words, if $\{g_n\}_{n \in \omega} $ enumerates the vertices of $\Gamma$ , for each $ n \in \omega $ we add a clique $K_n$ and connect it to $g_n$.
\begin{lemma}\label{notRigid}
\iGas is not rigid
\end{lemma}
\begin{proof}
for every $n \geq 2$ and for every two distinct  $a, b \in K_n$, there is an automorphism of $\Gamma^{*}$ swapping $a$ with $b$ and fixing all other vertices.
\end{proof}
\begin{lemma}\label{Id}
If $\sigma:\Ga \rightarrow \Ga$ is an automorphism of  \iGas then $\sigma \upharpoonright \G = Id_{\G}$
\end{lemma}
\begin{proof}
For each $n\in \omega$, $K_n$ is the set of vertices in $\Gamma^*$ of degree precisely $n$, and $g_n$ is the unique vertex of infinite degree connected to all vertices in $K_n$. Thus  $\sigma(g_n) = g_n$.
\end{proof}
\begin{proposition}\label{gammaStarNotSymIndivisible}
$\Gamma^*$ is indivisible but not symmetrically indivisible.
\end{proposition}

\begin{proof}
First, \iGas is indivisible, since clearly $\G \lesssim \Ga$ and by Fact \ref{gammaEmbeds}, $G \lesssim_s \G $ for every countable graph $G$. In particular $\Ga \lesssim \G$, thus $\Ga \sim \G $ and we have \iGs is indivisible thus, by Proposition  \ref{equivResult}, \iGas is indivisible. 

To show \iGas is not symmetrically indivisible, let $c:\Ga \rightarrow \{\ red ,\ blue\ \}$ be  a colouring of \iGas defined by 
\[ c(x) = \twopartdef { red } {x \in \G} { blue } {x \not\in \G} \]
 Since there is no blue vertex of infinite degree there is no blue copy of $\Gamma$, and therefore also $\Gamma^*$ does not embed in the blue sub-graph. Let ${\Ga}^{ \prime}$ be a red substructure isomorphic to $\Ga$. By Lemmas  \ref{notRigid} and \ref{Id}, ${\Ga}^{ \prime}$ has an automorphism that cannot be extended to an automorphism of \iGa , and so  ${\Ga}^{ \prime}$ is not symmetrically embedded in \iGa .
\end{proof}

\subsection{Enumeration endowments}\ 

Throughout this subsection fix an ultrahomogeneous structure $U$ in a relational language $\cL$ satisfying the pigeonhole property  and a structure $U^{*}$ not symmetrically indivisible such that $U\lesssim U^{*}\lesssim U$. (For example, the random graph $\Gamma$ and $\Gamma^{*}$ defined above.)

Before we continue our construction, we give a general claim about the pigeonhole property:
\begin{lemma}\label{extendInf}
If $a_1,\dots, a_n, b \in U$ are distinct and\ $g:\{a_1, \dots, a_n\}\to U$ is a partial isomorphism, then the substructure whose universe is
\[ S: =\Set{x\in U | g \cup \langle b, x\rangle \text{ is a partial isomorphism}} \] is isomorphic to $U$, in particular, $S$ is infinite.
\end{lemma}

\begin{proof}
By ultrahomogeneity, $S$ is non-empty and since $a_1,\dots, a_n,b$ are distinct, $g(a_1),\dots, g(a_n)\notin S$. Let $s\in S$ and let $\widehat{g} = g\cup \{\langle b, s\rangle\}$.
By the pigeonhole property, either $S$ or $U\setminus S$ is isomorphic to $U$.
Assume towards a contradiction that there is an isomorphism $\phi:U\to U\setminus S$. Then 
\[\phi^{-1}\upharpoonright \Set{\phi\circ g(a_1),\dots, \phi\circ g(a_n)}\]
is a partial isomorphism of $U\setminus S$, and thus by ultrahomogeneity, there is a $y\in U\setminus S$ such that
\[f: = \left(\phi^{-1}\upharpoonright \Set{\phi\circ g(a_1),\dots, \phi\circ g(a_n)}\right) \cup \{\langle \phi(s),y \rangle\}\]
is a partial isomorphism. Now $f\circ \phi \circ \widehat{g}$ is a partial isomorphism extending $g$ and $f\circ \phi \circ \widehat{g} (b)\in U\setminus S$, contradicting the definition of $S$.
\end{proof}

\begin{definition}
For an $\cL$ structure $\cM$, we say an expansion of $\cM$ to $\cL \cup \{<\}$ is an enumeration endowment if $<$ is of order type $\omega$.

Note that two enumeration endowments of the same structure, even in the ultrahomogeneous context, are not necessarily isomorphic.
\end{definition}

\begin{definition}
Recall that for a first-order relational structure $\cM$, $\age(\cM)$ is the class of all finite structures which are embeddable in $\cM$.
\end{definition}

\begin{lemma}\label{endowmentAgeEmbeds}
Let $A$ be a countable structure with $\age(A)\subseteq \age(U)$. If $U^<$,  $A^<$ are enumeration endowments of $U,A$ respectively, then $A^<$ embeds into $U^<$ (as $\cL\cup \{<\}$-structures).

In particular, $U^<$ is indivisible.
\end{lemma}
\begin{proof}
Let $\langle u_i : i\in \omega\rangle, \langle a_i : i\in \omega\rangle $ be enumerations of $U$ and $A$ respectively, compatible with the given enumeration endowments.

We construct the embedding inductively:
\begin{itemize}
\item Since $\age(A)\subseteq \age(U)$, there is a $u\in U$ such that $\langle a_0, u\rangle$ is a partial isomorphism. Let $e_0:=\langle a_0, u\rangle $ for such a $u$.
\item By Lemma \ref{extendInf}, for every $i\in \omega$, 
\[ S: =\Set{x\in U | e_i \cup \langle a_{i+1}, x\rangle \text{ is a partial $\cL$-isomorphism}} \]
is infinite. Choose $u\in S$ such that $e(a_i)<u$ in the enumeration endowment and let $e_{i+1} = e_i \cup \langle a_{i+1}, u\rangle$.
\end{itemize}
Let \[e:=\bigcup_{i\in \omega}e_i .\]

By the construction $e$ is an ascending union of $\cL$-partial isomorphisms, so it is an $\cL$-embedding. Furthermore it is order preserving -- thus it is an $\cL \cup \{<\}$-embedding.
\bigskip

Now to show $U^<$ is indivisible, let $c:U\to \{red, blue\}$.
By indivisibility of $U$ as an $\cL$-structure, there is a monochromatic $U^{\prime}$ $\cL$-isomorphic to $U$. Let $(U')^{<}$ be the induced $\cL \cup \{<\}$-structure on $U^{\prime}$. So $(U')^<$ is an enumeration endowment of $U'$ (not necessarily isomorphic to $U^<$). By the present lemma, $U^<$ embeds into $(U^{\prime})^<$ and it is monochromatic.

\end{proof}

\begin{remark} Note that since $\langle\omega,<\rangle$ is rigid, every enumeration endowment is rigid as well, thus in the context of enumerated graphs, symmetric indivisibility and indivisibility coincide.
\end{remark}

\begin{theorem}\label{symNotReduct}
A reduct of a symmetrically indivisible structure to a sub-language is not necessarily symmetrically indivisible.
\end{theorem}
\begin{proof}
Consider $U$ and $U^{*}$ above. Assume for simplicity $U\subseteq U^{*}$. Let $(U^{*})^<$ be an enumeration endowment of $U^{*}$ and let $U^<$ be the induced $\cL\cup \{<\}$-substructure on $U$.
Notice that $U^<$ is an enumeration endowment of $U$ and thus by Lemma  \ref{endowmentAgeEmbeds}, \[ (U^{*})^< \lesssim U^< \]
so by rigidity,
\[(U^{*})^< \lesssim_s U^<\lesssim_s (U^{*})^<\]
By Lemma  \ref{endowmentAgeEmbeds} and by rigidity $U^<$ is symmetrically indivisible and thus by Proposition  \ref{equivResult}, so is $(U^{*})^<$. But $ U^<\upharpoonright \cL = U^{*}$ is not symmetrically indivisible.
\end{proof}

\subsection*{Acknowledgement}
	The work in this paper is part of the author's M.Sc. thesis, prepared under the supervision of Assaf Hasson. The author would like to gratefully acknowledge him for presenting the question discussed in the paper, as well as the great help and support along the way. Thanks are also due to Menachem Kojman for his help in a preliminary version of the paper. The author was partially supported by an Israel Science Foundation grant number 1156/10.

\bibliographystyle{alpha}
\bibliography{manyref}
\end{document}